\documentclass[8pt,reqno]{amsproc}

\usepackage{amsmath}
\usepackage{amsthm}
\usepackage{amssymb}
\usepackage{graphicx}
\usepackage{hyperref}

\usepackage{mathpazo} %
\normalfont
\usepackage[T1]{fontenc}

\newcommand{\R}{\ensuremath{R}}
\newcommand{\Rep}{\ensuremath{R^\epsilon}}
\newcommand{\g}[1]{\ensuremath{g\left( #1 \right)}}
\newcommand{\gep}[1]{\ensuremath{g_\epsilon\!\left( #1 \right)}}
\newcommand{\HH}{\ensuremath{\mathcal{H}}}
\newcommand{\VV}{\ensuremath{\mathcal{V}}}
\newcommand{\Gep}{\ensuremath{\Gamma^\epsilon}}
\newcommand{\nablep}{\ensuremath{\nabla^\epsilon}}
\newcommand{\II}{\ensuremath{\mathrm{II}}}
\newcommand{\Pf}{\ensuremath{\mathrm{Pf}}}
\newcommand{\mrm}[1]{\mathrm{#1}}

\theoremstyle{plain}
\newtheorem*{theorem*}{Theorem}
\newtheorem{theorem}{Theorem}

\newtheorem*{proposition*}{Proposition}
\newtheorem{corollary}[theorem]{Corollary}

\theoremstyle{definition}

\newtheorem*{definition*}{Definition}
\newtheorem*{example*}{Example}
\newtheorem*{remark*}{Remark}

\title{Shrinking the Fibers of a Submersion Splits the Riemann Tensor}
\author{Carl McTague}
\email{mctague@math.jhu.edu}
\urladdr{\href{http://www.mctague.org/carl}{www.mctague.org/carl}}
\address{Mathematics Department, Johns Hopkins University, Baltimore, MD 21218, USA}



\begin{document}
\begin{abstract}
  This paper uses Karcher's formulation \cite{karcher-1999} of the O'Neill tensors \cite{oneill-1966,gray-1967} to derive a concise formula for the family $\Omega^\epsilon$ of curvature forms obtained by shrinking the fibers of a submersion $\pi:M\to B$ of semi-Riemannian manifolds by a factor of $1-\epsilon$. The formula clearly shows that as $\epsilon$ approaches~1, $\Omega^\epsilon$ approaches the sum of the vertical curvature form $\Omega^\mathrm{V}$ and the pullback $\pi^*\Omega^B$ of the curvature form of $B$. The Gauss-Bonnet integrand $\Pf(\Omega^\epsilon)$ therefore approaches the wedge $\Pf(\Omega^\mathrm{V})\wedge\pi^*\Pf(\Omega^B)$. So if $\pi$ has compact fiber $F$, the pushforward $\pi_*\Pf(\Omega^\epsilon)$ approaches $\chi(F)\cdot\Pf(\Omega^B)$.
\end{abstract}

\maketitle

\section{Karcher's Formulation of the O'Neill Tensors}

A \emph{submersion $\pi:(M,g)\to(B,g_B)$ of semi-Riemannian manifolds} is a smooth map whose derivative $\mrm{D}\pi$ restricts to an isometry:
\begin{align*}
  \mrm{D}\pi|_{\mrm{H}M}:\mrm{H}M\to\mrm{T}B
\end{align*}
from the \emph{horizontal bundle} $\mrm{H}M$, i.e.\ the orthogonal complement of the \emph{vertical bundle} $\mrm{V}M=\ker \mrm{D}\pi\subset\mrm{T}M$, to the tangent bundle $\mrm{T}B$ of $B$.

The O'Neill tensors \cite{oneill-1966,gray-1967} are to a submersion what the second fundamental form is to an immersion. Karcher \cite{karcher-1999} elegantly formulated them in terms of the covariant derivatives of the orthogonal projections:
\begin{align*}
  \HH,\VV&:\mrm{T}M\to\mrm{T}M
\end{align*}
onto $\mrm{H}M$ and $\mrm{V}M$. This approach clarifies the tensors' symmetries and reduces the number of cases to be considered. (Since $\HH+\VV=\mrm{id}$, and therefore $\nabla\HH=-\nabla\VV$, we can mostly avoid $\VV$.)

We will rely on Karcher's formulas so heavily that we begin by restating them, prefixing a `K' to his numbering to ease cross-referencing. In a few cases we improve on his formulas and, to ease comparison, prefix an `M' to his numbering. The rest of this paper's formulas are numbered to avoid clashing with Karcher's.

\medskip
Let $X,Y,Z$ be sections of $\mrm{T}M$; $U,V,W$ sections of $\mrm{V}M$; and $H,K,L$ sections of $\mrm{H}M$. \\
The horizontal and vertical components of the connection $\nabla$ of $M$:
\begin{align*}
  \tag{K$1_\mrm{H}$}
  \HH\cdot\nabla_XV &= -\nabla_X\HH\cdot V \\
  \tag{K$1_\mrm{V}$}
  \VV\cdot\nabla_XH &= +\nabla_X\HH\cdot H
\end{align*}
Therefore, if $\nabla^\mrm{H}$ and $\nabla^\mrm{V}$ are the induced connections on $\mrm{H}M$ and $\mrm{V}M$ then:
\begin{align*}
  \tag{K$2_\mrm{H}$}
  \nabla_XH &=\nabla^\mrm{H}_XH + \nabla_X\HH\cdot H \\
  \tag{K$2_\mrm{V}$}
  \nabla_XV &=\nabla^\mrm{V}_XV - \nabla_X\HH\cdot V
\end{align*}
Compatibility with metric:
\begin{align*}
  \tag{K$3_1$} 
  \begin{aligned}
  \g{\HH\cdot Y,Z} &= \g{Y,\HH\cdot Z} \\
  \g{\nabla_X\HH\cdot Y,Z} &= \g{Y,\nabla_X\HH\cdot Z} \\
  g\big(\nabla^2_{X,X'}\HH\cdot Y,Z\big) &= g\big(Y,\nabla^2_{X,X'}\HH\cdot Z\big)
  \end{aligned}
\end{align*}
$\nabla_Y\HH$ maps $\mrm{V}M\to\mrm{H}M$ and $\mrm{H}M\to\mrm{V}M$; in fact:
\begin{align*}
  \tag{K$3_2$}
  \nabla_Y\HH\cdot\VV &= \HH\cdot\nabla_Y\HH &
  \nabla_Y\HH\cdot\HH &= \VV\cdot\nabla_Y\HH
\end{align*}
A Codazzi-like equation and its restriction to $\mrm{H}M$ and $\mrm{V}M$:
\begin{align*}
  \tag{K$3_3$}
  \begin{aligned}
    \nabla^2_{X,Y}\HH\cdot\VV-\nabla_Y\HH\cdot\nabla_X\HH &= \nabla_X\HH\cdot\nabla_Y\HH+\HH\cdot\nabla^2_{X,Y}\HH \\
    -\HH\cdot\nabla^2_{X,Y}\HH\cdot H &= \big(\nabla_X\HH\cdot\nabla_Y\HH+\nabla_Y\HH\cdot\nabla_X\HH\big)\cdot H  \\
    +\VV\cdot\nabla^2_{X,Y}\cdot V &= \big(\nabla_X\HH\cdot\nabla_Y\HH+\nabla_Y\HH\cdot\nabla_X\HH\big)\cdot V
  \end{aligned}
  \end{align*}
Symmetry of $\nabla\HH$ on $\mathrm{V}M$:
    \begin{align*}
      \nabla_U\HH\cdot V - \nabla_V\HH\cdot U = 0 \tag{K$6_3$}
    \end{align*}
Skewsymmetry of $\nabla\HH$ on $\mathrm{H}M$:
  \begin{align*}
    \nabla_H\HH\cdot K + \nabla_K\HH\cdot H =0 \tag{Prop.~K3b}
  \end{align*}
Decomposition of the Riemann tensor:
\begin{align*}
  \R(X,Y)V &= -\overbrace{\R(X,Y)\HH\cdot V}^{\text{in $\mrm{H}M$}} \,+ \overbrace{\R^{\mrm{V}}(X,Y)V\,-\big[\nabla_X\HH,\nabla_Y\HH\big]V}^{\text{in $\mrm{V}M$}} \tag{K10, K11} \\
  \R(X,Y)H &= \phantom{-}\underbrace{\R(X,Y)\HH\cdot H}_{\text{in $\mrm{V}M$}} + \underbrace{\R^{\mrm{H}}(X,Y)H-\big[\nabla_X\HH,\nabla_Y\HH\big] H}_{\text{in $\mrm{H}M$}} \tag{K10, K12}
\end{align*}
If $X,Y$ are vertical then (K11) is the Gauss equation of the fibers.\footnote{Indeed, by (K11) and (K$3_1$) and (K$2_\mrm{V}$):
\begin{align*}
  \g{\R(X,Y)U,V} &=
  g(\R^{\mrm{V}}(X,Y)U,V) - \g{\nabla_X\HH\cdot U,\nabla_Y\HH\cdot V} + \g{\nabla_Y\HH\cdot U,\nabla_X\HH\cdot V} \\
  &=g(\R^{\mrm{V}}(X,Y)U,V) - \g{\II(X,U),\II(Y,V)} + \g{\II(Y,U),\II(X,V)}
\end{align*}
}
Beware that $R(X,Y)\HH\cdot Z$ does \emph{not} satisfy a cyclic Bianchi identity in $X,Y,Z$.

\section{Shrinking the Fibers}

Now shrink the fibers of $\pi$ by a factor of $1-\epsilon$:
\begin{align*}
  \gep{X,Y} &= g\big((1-\epsilon\VV)\cdot X, (1-\epsilon\VV)\cdot Y\big)
\end{align*}
Thus:
\begin{align*}
  \gep{X,H} &= g(X,H) &
  \gep{X,V} &= (1-\epsilon)^2\cdot g(X,V)
\end{align*}

Let $\nablep$ be the Levi-Civita connection of $(M,g_\epsilon)$ and let $\Gep$ be the difference tensor:
\begin{align*}
  \Gep(X,Y)=\nablep_XY-\nabla_XY
\end{align*}

Karcher's equation (K15) may be written:
\begin{align*}
  (\nabla_Xg_\epsilon)(Y,Z) = \epsilon(2-\epsilon)\cdot g(Y,\nabla_X\HH\cdot Z) \tag{M15}
\end{align*}

Karcher's equations (K16) extend to a single general formula:
\begin{align*}
  \Gep(X,Y) = \epsilon(2-\epsilon)\cdot\big(\nabla_{\HH X}\HH\cdot\VV Y +\nabla_{\HH Y}\HH\cdot\VV X +\nabla_{\VV X}\HH\cdot\VV Y\big) \tag{M$16_1$}
\end{align*}
The first and third terms combine easily but $\Gep$---being the difference of torsion free connections---is \emph{symmetric}, and this three-term formula showcases that symmetry---the third term being symmetric according to (K$6_3$). Observe that \emph{$\Gep$ is always horizontal}. Some useful special cases:
\begin{align*}
  \tag{M$16_2$}
  \begin{aligned}
    \Gep(X,V) &= \epsilon(2-\epsilon)\cdot\nabla_X\HH\cdot V \\
    \Gep(X,H) &= \epsilon(2-\epsilon)\cdot\nabla_H\HH\cdot\VV X
  \end{aligned}
\end{align*}

\newcommand{\npiGe}{\ensuremath{18}}
Comparing (K9) and (K16) leads to:
\begin{align*}
  \epsilon(2-\epsilon)\cdot(\nabla^2\pi)(X,Y) &= \mrm{D}\pi\cdot\Gep(X,Y) \tag{\npiGe}
\end{align*}

\newcommand{\neHH}{\ensuremath{19}}

Some additional formulas worth recording:
\begin{align*}
  \tag{\neHH}
  \begin{aligned}
  \nablep_X\HH\cdot H &= \nabla_X\HH\cdot H \\
  \nablep_X\HH\cdot V &= (1-\epsilon)^2\cdot\nabla_X\HH\cdot V \\
  \nablep_X\HH\cdot \nablep_Y\HH &= (1-\epsilon)^2\cdot\nabla_X\HH\cdot \nabla_Y\HH \\
  \big[\nabla^\epsilon_X\HH,\nabla^\epsilon_Y\HH\big] &= (1-\epsilon)^2\cdot \big[\nabla_X\HH,\nabla_Y\HH\big]
  \end{aligned}
\end{align*}

\section{Effect on the Riemann Tensor}

Let $V_1,\dots,V_k,H_{k+1},\dots,H_n$ be a positively-oriented orthonormal moving frame on $M$ consisting of vertical followed by horizontal vectors. Let $V^\epsilon_i=\tfrac1{1-\epsilon}V_i$. Then $V^\epsilon_1,\dots,V^\epsilon_k,H_{k+1},\dots,H_n$ is a positively-oriented orthonormal moving frame on $(M,g_\epsilon)$.

\newcommand{\Rdefo}[1]{\ensuremath{20_#1}}
\begin{theorem}
  \label{thm:defo}
  \begin{align*}
    \tag{\Rdefo1} 
      g_\epsilon\big(\Rep(X,Y)V^\epsilon_i,V^\epsilon_j\big)
    &= g\big(\R^{\mrm{V}}(X,Y)V_i,V_j\big)-(1-\epsilon)^2\cdot g\big([\nabla_X\HH,\nabla_Y\HH] V_i,V_j\big) \\
    \tag{\Rdefo2} 
    g_\epsilon\big(\Rep(X,Y)H_i,V^\epsilon_j\big) &=
        (1-\epsilon)\cdot \Big[ g\big(\R(X,Y)\HH\cdot H_i,V_j\big) \\
        &\hspace{6em} + \epsilon(2-\epsilon)\cdot g\big(\nabla_X\HH\cdot\nabla_{H_i}\HH\cdot Y - \nabla_Y\HH\cdot\nabla_{H_i}\HH\cdot X,V_j\big) \Big] \\
    \tag{\Rdefo3}
    g_\epsilon\big(\Rep(X,Y)V^\epsilon_i,H_j\big) &= (\epsilon-1)\cdot \Big[ g\big(\R(X,Y)\HH\cdot V_j,H_i\big) \\&\hspace{6em}+ \epsilon(2-\epsilon)\cdot g\big(\nabla_{\nabla_Y\HH\cdot V_j}\HH\cdot X - \nabla_{\nabla_X\HH\cdot V_j}\HH\cdot Y,H_i\big)\Big] \\
  \tag{\Rdefo4}
  g_\epsilon\big(\Rep(X,Y)H_i,H_j\big) &= \epsilon(2-\epsilon)\cdot g^B\big(R^B\big(\mrm{D}\pi\cdot X,\mrm{D}\pi\cdot Y\big)\big(\mrm{D}\pi\cdot H_i\big),\mrm{D}\pi\cdot H_j\big) \\ &\hspace{6em}+ (1-\epsilon)^2\cdot g\big(R(X,Y)H_i,H_j\big)
  \end{align*}
\end{theorem}

\begin{remark*}
  (\Rdefo2) and (\Rdefo3) are opposite as expected since:
  \begin{align*}
    \g{\nabla_{\nabla_Y\HH\cdot V}\HH\cdot X,H} &\overset{(\text{K$3_1$)}}{=} \g{X,\nabla_{\nabla_Y\HH\cdot V}\HH\cdot H} \\ &\overset{\text{Prop.}}{\underset{\text{K3b}}{=}} \g{X,-\nabla_H\HH\cdot \nabla_Y\HH\cdot V} \overset{\text{(K$3_1$)}}{=} -\g{\nabla_Y\HH\cdot\nabla_H\HH\cdot X,V}
  \end{align*}
\end{remark*}

\begin{corollary}
  Relative to $V^\epsilon_1,\dots,V^\epsilon_k,H_{k+1},\dots,H_n$, the curvature form of $(M,g_\epsilon)$ looks like:
  \begin{align*}
    \Omega^\epsilon
    =
  \left[
    \begin{array}{c|c}
      \Omega^\mathrm{V} + O(1-\epsilon)^2 & O(1-\epsilon) \\ \hline
      O(1-\epsilon) & \pi^*(\Omega^B) + O(1-\epsilon)^2
    \end{array}
  \right]
  \end{align*}
  where $\pi^*(\Omega^B)$ is the pullback of the curvature form of $B$. So as $\epsilon\to1$:
  \begin{align*}
    \Pf(\Omega^\epsilon) &\to \Pf(\Omega^\mathrm{V}) \wedge \pi^*\big(\Pf(\Omega^B)\big)
  \end{align*}
  If $\pi$ has compact fiber $F$ then integrating along the fibers gives:
  \begin{align*}
    \pi_*\big(\Pf(\Omega^\epsilon)\big) &\to \chi(F)\cdot\Pf(\Omega^B)
  \end{align*}
\end{corollary}

\section{Proof of Theorem~\ref{thm:defo}}

$\bullet$ {(\Rdefo1)} is easiest:
\begin{align*}
  \VV\cdot\Rep(X,Y)V &\overset{\text{(K11)}}{=} \R^\mrm{V}(X,Y)V - [\nablep_X\HH,\nablep_Y\HH]V \\
  &\overset{(\neHH)}= \R^\mrm{V}(X,Y)V - (1-\epsilon)^2\cdot[\nabla_X\HH,\nabla_Y\HH]V
\end{align*}

\smallskip
\newcommand{\ReHH}[1]{\ensuremath{21_#1}}
$\bullet$ {(\Rdefo2)} and {(\Rdefo3)} follow from (K10) and:
\begin{proposition*}
  \begin{align*}
    \Rep(X,Y)\HH\cdot H &= \R(X,Y)\HH\cdot H + \epsilon(2-\epsilon)\cdot\big( \nabla_X\HH\cdot\nabla_H\HH\cdot\VV Y - \nabla_Y\HH\cdot\nabla_H\HH\cdot\VV X \big) \tag{\ReHH1} \\
    \Rep(X,Y)\HH\cdot V &= (1-\epsilon)^2\cdot\Big( \R(X,Y)\HH\cdot V + \epsilon(2-\epsilon)\cdot\big(\nabla_{\nabla_Y\HH\cdot V}\HH\cdot \VV X - \nabla_{\nabla_X\HH\cdot V}\HH\cdot \VV Y \big)\Big) \tag{\ReHH2}
  \end{align*}
\end{proposition*}
\begin{proof}%

  Since:
  \begin{align*}
    \nabla^2_{X,Y}\HH\cdot H &= \nabla_X(\nabla_Y\HH\cdot H) - \nabla_{\nabla_XY}\HH\cdot H - \nabla_Y\HH\cdot\nabla_XH
\intertext{it follows by (K$3_2$), (\neHH), $\VV\cdot\Gep=0$ and (M$16_2$) that:}
    \VV\cdot(\nablep)^2_{X,Y}\HH\cdot H &= \VV\cdot\Big( \nablep_X(\nablep_Y\HH\cdot H)-\nablep_{\nablep_XY}\HH\cdot H-\nablep_Y\HH\cdot(\HH\cdot\nablep_XH)\Big) \\
    &= \VV\cdot\Big( \nabla_X(\nabla_Y\HH\cdot H)-\nabla_{\nabla_XY+\Gep(X,Y)}\HH\cdot H-\nabla_Y\HH\cdot\HH\cdot(\nabla_XH+\Gep(X,H))\Big) \\
    &= \VV\cdot\nabla^2_{X,Y}\HH\cdot H - \nabla_{\Gep(X,Y)}\HH\cdot H - \epsilon(2-\epsilon)\cdot\nabla_Y\HH\cdot\nabla_H\HH\cdot\VV X
  \intertext{Since $\Gep$ is symmetric, the second term drops out of:}
    \VV\cdot\Rep(X,Y)\HH\cdot H &= \VV\cdot\R(X,Y)\HH\cdot H + \epsilon(2-\epsilon)\cdot\big(\nabla_X\HH\cdot\nabla_H\HH\cdot\VV Y - \nabla_Y\HH\cdot\nabla_H\HH\cdot\VV X \big)
  \end{align*}
  (\ReHH1) follows since the symmetry of (K$3_3$) restricted to $\mrm{H}M$ implies that:
  \begin{align*}
    \HH\cdot\Rep(X,Y)\HH\cdot H=0=\HH\cdot\R(X,Y)\HH\cdot H
  \end{align*}

  Similarly, since:
  \begin{align*}
    \nabla^2_{X,Y}\HH\cdot V &= \nabla_X(\nabla_Y\HH\cdot V) - \nabla_{\nabla_XY}\HH\cdot V - \nabla_Y\HH\cdot\nabla_XV
\intertext{it follows by (K$3_2$), (\neHH), $\VV\cdot\Gep=0$ and (M$16_2$) that:}
    \HH\cdot(\nablep)^2_{X,Y}\HH\cdot V &= \HH\cdot\Big( \nablep_X(\nablep_Y\HH\cdot V)-\nablep_{\nablep_XY}\HH\cdot V-\nablep_Y\HH\cdot\nablep_XV\Big) \\
    &= (1-\epsilon)^2\cdot\HH\cdot\Big( \nablep_X(\nabla_Y\HH\cdot V)-\nabla_{\nablep_XY}\HH\cdot V - \nabla_Y\HH\cdot(\VV\cdot\nablep_XV)\Big) \\
    &= (1-\epsilon)^2\cdot\HH\cdot\Big( \nabla_X(\nabla_Y\HH\cdot V)+\Gep(X,\nabla_Y\HH\cdot V)-\nabla_{\nabla_XY+\Gep(X,Y)}\HH\cdot V - \nabla_Y\HH\cdot(\VV\cdot\nabla_XV)\Big) \\
    &= (1-\epsilon)^2\cdot \Big( \HH\cdot\nabla^2_{X,Y}\HH\cdot V + \epsilon(2-\epsilon)\cdot\nabla_{\nabla_Y\HH\cdot V}\HH\cdot\VV X - \nabla_{\Gep(X,Y)}\HH\cdot V \Big)
    \intertext{Since $\Gep$ is symmetric, the third term drops out of:}
    \HH\cdot\Rep(X,Y)\HH\cdot V &= (1-\epsilon)^2\cdot \Big( \HH\cdot\R(X,Y)\HH\cdot V + \epsilon(2-\epsilon)\cdot\big( \nabla_{\nabla_Y\HH\cdot V}\HH\cdot\VV X - \nabla_{\nabla_X\HH\cdot V}\HH\cdot\VV Y\big)\Big)
  \end{align*}
  (\ReHH2) follows since the symmetry of (K$3_3$) restricted to $\mrm{V}M$ implies that:
  \begin{align*}
    \VV\cdot\Rep(X,Y)\HH\cdot V&=0=\VV\cdot\R(X,Y)\HH\cdot V\qedhere
  \end{align*}
\end{proof}

\smallskip
\newcommand{\DpiRe}{\ensuremath{22}}
$\bullet$ {(\Rdefo4)} follows from:
\begin{proposition*}
\begin{align*}
\mrm{D}\pi\cdot \Rep(X,Y)Z = \epsilon(2-\epsilon)\cdot \R^B(\mrm{D}\pi\cdot X,\mrm{D}\pi\cdot Y)(\mrm{D}\pi\cdot Z) + (1-\epsilon)^2\cdot \mrm{D}\pi\cdot\R(X,Y)Z \tag{\DpiRe}
\end{align*}
\end{proposition*}
(This extends the last result of Karcher's paper, which asserts this for $\epsilon=1$ and $X,Y,Z$ horizontal.)
\begin{proof}%
\newcommand{\dGe}{\ensuremath{23}}

Differentiating (\npiGe) gives:
\begin{align*}
  \epsilon(2-\epsilon)\cdot(\nabla^3\pi)(X,Y,Z) = \mrm{D}\pi \big[ (\nabla_X\Gep)(Y,Z) + \Gep(X,\Gep(Y,Z)) \big] \tag{\dGe}
\end{align*}
A straightforward calculation shows that:
\newcommand{\ReGe}{\ensuremath{24}}
\begin{align*}
\Rep(X,Y)Z = R(X,Y)Z + (\nabla_X\Gep)(Y,Z) + \Gep(X,\Gep(Y,Z)) - (\nabla_Y\Gep)(X,Z) - \Gep(Y,\Gep(X,Z)) \tag{\ReGe}
\end{align*}
Taken together:
\begin{align*}
\mrm{D}\pi\cdot \Rep(X,Y)Z = \mrm{D}\pi\cdot R(X,Y)Z + \epsilon(2-\epsilon)\cdot\big[(\nabla^3\pi)(X,Y,Z) - (\nabla^3\pi)(Y,X,Z)\big]
\end{align*}
\newcommand{\ntp}{\ensuremath{25}}
(\DpiRe) follows since:
  \begin{align*}
    (\nabla^3\pi)(X,Y,Z)-(\nabla^3\pi)(Y,X,Z)=R^B(\mrm{D}\pi\cdot X,\mrm{D}\pi\cdot Y)(\mrm{D}\pi\cdot Z) - \mrm{D}\pi\cdot R(X,Y)Z \tag{\ntp}
  \end{align*}
\end{proof}

\section*{Acknowledgements}

Thanks to Hermann Karcher for his paper and for an encouraging correspondence.

\bigskip
\begin{minipage}{0.8\textwidth}
\it ``Come! my head's free at last!'' said Alice in a tone of delight, which changed into alarm in another moment, when she found that her shoulders were nowhere to be seen: she looked down upon an immense length of neck, which seemed to rise like a stalk out of a sea of green leaves that lay far below her.

\medskip
``What \emph{can} all that green stuff be?'' said Alice, ``and where \emph{have} my shoulders got to? And oh! my poor hands! how is it I ca'n't see you?'' She was moving them about as she spoke, but no result seemed to follow, except a little rustling among the leaves. Then she tried to bring her head down to her hands, and was delighted to find that her neck would bend about easily in every direction, like a serpent. 

\medskip
\hfill \rm ---from Ch.~4 of \href{https://archive.org/stream/AlicesAdventuresUnderGround1864#page/n63/mode/2up}{\emph{Alice's Adventures Under Ground}} by Lewis Carroll (1864)
\end{minipage}
\begin{minipage}{10pt}
\includegraphics[height=60mm]{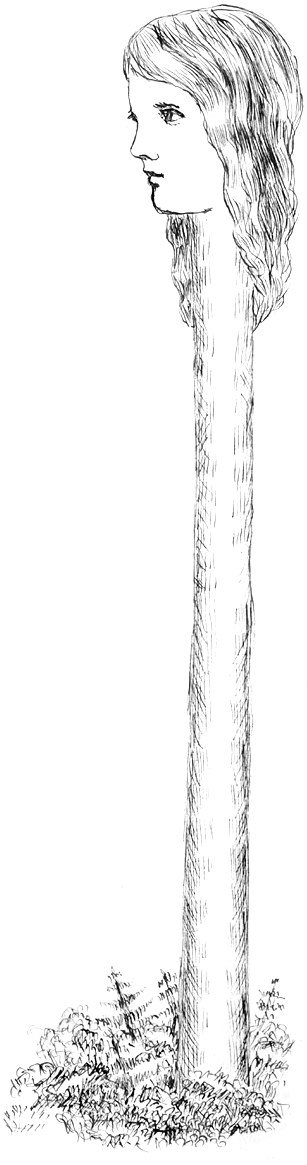}
\end{minipage}

\def\cprime{$'$}


\begin{thebibliography}{O'N66}

\bibitem[Gra67]{gray-1967}
Alfred Gray.
\newblock Pseudo-{R}iemannian almost product manifolds and submersions.
\newblock {\em J. Math. Mech.}, 16:715--737, 1967.

\bibitem[Kar99]{karcher-1999}
H.~Karcher.
\newblock Submersions via projections.
\newblock {\em Geom. Dedicata}, 74(3):249--260, 1999.

\bibitem[O'N66]{oneill-1966}
Barrett O'Neill.
\newblock The fundamental equations of a submersion.
\newblock {\em Michigan Math. J.}, 13:459--469, 1966.

\end{thebibliography}
\end{document}